\documentclass[imslayout,preprint,amsmathload,amsthmload,natbibload]{imsart}
\usepackage{amsthm,amsmath}
\RequirePackage{natbib} 
\usepackage{graphicx}
\usepackage{mathptmx}
\usepackage{amssymb}
\usepackage[dvipsnames]{xcolor}
\usepackage{dlfltxbcodetips} 
\usepackage{array}

\usepackage{macros}

\startlocaldefs
\DeclareMathOperator{\inv}{in}
\DeclareMathOperator{\out}{out}
\DeclareMathOperator{\neigh}{N}
\DeclareMathOperator{\syz}{Syz}
\newcommand{\arcs}{\mathcal A}
\newcommand{\cycles}{\mathcal C}
\newcommand{\edges}{\mathcal E}
\newtheorem{example}{Example}
\newtheorem{proposition}{Proposition}
\newtheorem{theorem}{Theorem}
\newtheorem{remark}{Remark}
\newtheorem{definition}{Definition}
\endlocaldefs

\begin{document}

\begin{frontmatter}

\title{The algebra of reversible Markov chains}
\runtitle{Algebra of  reversible MC}


\author{\fnms{Giovanni} \snm{Pistone}\corref{}\ead[label=e1]{giovanni.pistone@carloalberto.org}\thanksref{t1}}
\thankstext{t1}{G. Pistone was supported by Collegio Carlo Alberto, Moncalieri}
\address{Via Real Collegio 30\\ 10024 Moncalieri, Italy \\ \printead{e1}}
\affiliation{Collegio Carlo Alberto}
\and
\author{\fnms{Maria Piera} \snm{Rogantin}\ead[label=e2]{rogantin@dima.unige.it}}
\address{Dipartimento di Matematica, Via Dodecaneso 30 \\ Genova, Italy \\ \printead{e2}}
\affiliation{Universit\`a di Genova}

\runauthor{G. Pistone and M.P. Rogantin}

\begin{abstract}
For a Markov chain both the detailed balance condition and the cycle Kolmogorov condition are algebraic binomials. This remark suggests to study reversible Markov chains with the tool of Algebraic Statistics, such as toric statistical models. One of the results of this study in an algebraic parameterization of reversible Markov transitions and their invariant probability. \end{abstract}



\end{frontmatter}

\section{Introduction}
On a finite state space $V$ we consider q-reversible (quasi-reversible) Markov matrices, i.e. Markov matrices $P$ with elements denoted $P_{v\to w}$, $v, w \in V$, such that $P_{v\to w}=0$ if, and only if, $P_{w\to v}=0$, $v\ne w$.

\begin{figure}[t]
  \begin{center}
    \includegraphics[scale=.5]{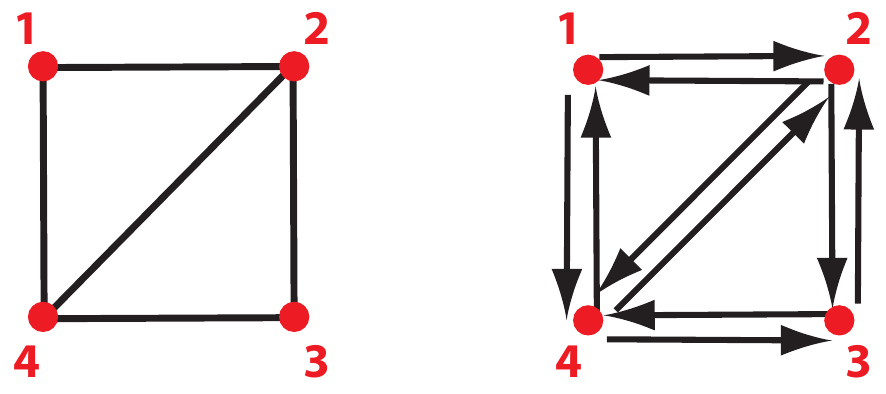}
  \end{center}
  \caption{Running example: undirected graph $\mathcal G$ and directed graph $\mathcal D$\label{fig:2arcs}}
\end{figure}
The support of $P$ is the graph $\mathcal G = (V,\edges)$, where $\overline{vw}$, $v\ne w$, is an edge if, and only if, $P_{v\to w}$ and $P_{w \to v}$ are both positive. We associate to each edge $\overline{vw}$ the two directed arcs $v\to w$ and $w\to v$ to get a directed graph without loops (i.e. arcs from $v$ to $v$) that we denote by $\mathcal D = (V,\arcs)$, see Figure \ref{fig:2arcs}. The neighborhood of $v$ is $\neigh(v)$, the degree of $v$ is $d(v)$, the set of arcs leaving $v$ is $\out(v)$, the set of arcs entering $v$ is $\inv(v)$.

Viceversa, given a connected graph $\mathcal G = (V,\edges)$, a q-reversible Markov matrix $P$ with structure $\mathcal G$ is such that $P_{v\to w} = 0$ if $\overline{vw} \notin \edges$ and $P_{v\to w}=0$ if, and only if, $P_{w\to v}=0$, $\overline{vw}\in\edges$. Such a Markov matrix is characterized as a mapping $P\colon \arcs \to \reals_+$ such that $\sum_{w\in\neigh(v)} P_{v\to w} \le 1$, $v \in V$. As we assume $P_{v\to w} = 0$, $\overline{vw}\notin\edges$, we have $P_{v\to v} = 1 - \sum_{w\in\neigh(v)} P_{v\to w}$. The support of $P$ will be a sub-graph of $\mathcal G$. We are going to need this detailed classification in the following.  For the time being, we note that the set of $P$'s is parameterized by a product of symplexes $\bigtimes_{v\in V} S(\neigh(v))$, where $S(I)=\setof{x\in\reals^I}{\sum_{i\in I} x_i \le 1}$ denotes the (solid) simplex over the index set $I$.

A Markov matrix $P$ on $V$ satisfies the \emph{detailed balance}  conditions if there exists $\kappa(v) > 0$, $v \in V$, such that
\begin{equation*}
  \kappa(v)P_{v\to w} = \kappa(w)P_{w\to v}, \quad v, w \in V.
\end{equation*}
It follows that $P$ is q-reversible and that $\pi(v)=\kappa(v)/\sum_{v\in V} \kappa(v)$ is an invariant probability with full support. Equivalently, the Markov chain $(X_n)$, $n=0,1,\dots$, with invariant probability $\pi$ and transition matrix $[P_{v\to w}]$, $v, w \in V$, has \emph{reversible} bivariate joint distribution
\begin{equation}\label{eq:2-reversible}
\probof{X_n = v, X_{n+1} = w} = \probof{X_n = w,X_{n+1}=v}, \quad v, w \in V, \quad n \ge 0.
\end{equation}
Such a Markov chain, and its transition matrix are called reversible. Reversible Markov Chains (MCs) are relevant in Statistical Physics, e.g. in the theory of entropy production, and in Applied Probability, e.g. the simulation method Monte Carlo Markov Chain (MCMC). The main aim of this paper is to find useful parameterizations of the reversible Markov matrices of a given structural graph.

In Section \ref{sec:BGD} we review some basics from \cite{dobrushin|sukhov|fritts:1988}, \cite{kelly:1979}, \cite[Ch 5]{strook:2005IMPs}, \cite{diaconis|rolles:2006}, \cite{hastings:1970mcmc}, \cite{peskun:1973}, \cite{liu:2008strategies}.

In Section \ref{sec:algebra} we discuss the algebraic theory prompted by the detailed balance condition. The results
pertain to the area of Algebraic Statistics, see e.g. \cite{pistone|riccomagno|wynn:2001},
\cite{drton|sturmfels|sullivan:2009}, \cite{gibilisco|riccomagno|rogantin|wynn:2010}. Previous results in the same
algebraic spirit were presented in \cite{suomela:1979} and \cite{mitrophanov:2004}. We believe the results here are
new; some proofs depend on classical notions of graph theory that are reviwed in some detail because of our particular
context.

The discussion and the conclusions are briefly presented in Section \ref{sec:discussion}.

\section{Background}\label{sec:BGD}
\subsection{Reversible Markov process}
The reversibility of the bivariate joint distribution in Equation \eqref{eq:2-reversible} gives a second parameterization of reversible Markov chains. In fact, the stochastic process $\left(X_n\right)_{n \ge 0}$ with state space $V$ is 2-reversible and 2-stationary if, and only if, Equation \eqref{eq:2-reversible} holds. When the process is a Markov chain, the distribution depends on the bivariate distributions only. In particular, the process is 1-stationary: by summing over $w \in V$, we have
\begin{equation*}
  \probof{X_n = v} = \probof{X_{n+1} = v} = \pi(v), \quad v \in V, n \ge 0.
\end{equation*}

Let $V_2$ be the set of all subset of $V$ of cardinality 2. The elements of $V_2$ are the edges of the full graph on $V$. The following parameterization of the 2-dimensional distributions has been used in \cite{diaconis|rolles:2006}:
\begin{align*}
\theta_{v} &= \probof{X_n = v, X_{n+1} = v}, \quad v \in V, \\
\theta_{\overline{vw}} &= \probof{X_n = v,X_{n+1}=w} + \probof{X_n=w,X_{n+1} = v} \\ &= 2\probof{X_n = v,X_{n+1}=w}, \quad \overline{vw} \in V_2.
\end{align*}
The number of parameters is $N+\binom N 2 = \binom {N+1} 2$; moreover it holds
\begin{equation*}
  1=  \sum_{v,w \in V} \probof{X_n = v,X_{n+1} = w}=  \sum_{v \in V} \theta_{v} + \sum_{\overline{vw} \in V_2} \theta_{\overline{vw}},
\end{equation*}
hence $\theta = (\theta_V, \theta_{V_2})$ belongs to $\Delta(V \cup V_2)$, where
$\Delta(I)=\setof{x\in\reals^I}{\sum_{i\in I} x_i = 1}$ denotes the (flat) simplex on the index set $I$.

Given an undirected graph $\mathcal G = (V,\edges)$ such that $\probof{X_n = v, X_{n+1} = w} = 0$ if $\overline{vw} \notin \edges$, then the vector of parameters $\theta=(\theta_V,\theta_\edges)$ belongs to the convex set $\Delta(V \cup \edges)$. We note that the vertices $V$ are identified with loops of the transitions because $\theta_v=\probof{X_n=v,X_{n+1}=v}$.

The marginal probability $\pi$ can be written using the  $\theta$ parameters:
\begin{equation*}
  \pi(v)= \sum_{w \in V} P\left(X_n = v,X_{n+1} = w\right)= \theta_{v} + \frac 1 2 \sum_{w \in\neigh(v)}\theta_{\overline{vw}},
\end{equation*}
or,  in matrix form,
\begin{equation*}
  \pi= \theta_{V} + \frac 12 \Gamma \theta_{\edges},
\end{equation*}
where $\Gamma$ is the incidence matrix of the graph $\mathcal G$.
\begin{example}[Running example]\label{ex:running}
Consider the graph $\mathcal G = (V,\edges)$ with $V=\set{1,2,3,4}$ and $\edges=\set{\overline{12}, \overline{23},
\overline{34},\overline{14}, \overline{24}}$, see left side of Figure \ref{fig:2arcs}.
Here
\begin{equation*}
  \Gamma=
\bordermatrix[{[}{]}]{
 & \overline{12} & \overline{23} & \overline{34} &\overline{14} & \overline{24}  \cr
1& 1 &  0 &  0 &  1 &  0\cr 2& 1 &  1 &  0 &  0 &  1\cr 3& 0 &  1 &  1 &  0 &  0\cr 4& 0 &  0 &  1 &  1 &  1\cr
 }.
\end{equation*}
\end{example}

All $\pi$'s are obtained this way if we admit positive probability on loops:
\begin{proposition}
\begin{enumerate}
\item The map
\begin{equation*}
  \gamma\colon \Delta(V \cup \edges)  \ni \theta=
  \begin{bmatrix}
    \theta_{V} \\ \theta_{\edges}
  \end{bmatrix}
  \longmapsto \pi =
  \begin{bmatrix}
    I_{V} & \frac 12 \Gamma
  \end{bmatrix}
  \begin{bmatrix}
    \theta_{V} \\ \theta_{\edges}
  \end{bmatrix}
  \in \Delta(V)
\end{equation*}
is a surjective Markov map.
\item The image of $(0,\theta_{\edges})$, $\theta_{\edges} \in \Delta(\edges)$, is the convex hull of the half points of each edge of the simplex $\Delta(V)$ whose vertices are connected in $\mathcal G$.
\end{enumerate}
\end{proposition}
\begin{proof}
\begin{enumerate}
\item Each probability $\pi$ is the image of $\theta=(\pi,0_\edges)$.
\item If $\theta_V=0$, then $\pi = \sum_{w\in\neigh(v)} \frac12 \Gamma_{\overline{vw}} \theta_{\overline{vw}}$, where $\sum_{w\in\neigh(v)} \theta_{\overline{vw}}=1$ and $\Gamma_{\overline{vw}}$ is the $\overline{vw}$-column of $\Gamma$. Hence, $\frac12\Gamma_{\overline{vw}}$ is the the middle point of the $v$- and the $w$-vertex of the simplex $\Delta(V)$.
\end{enumerate}
\end{proof}
Item 1 of the proposition leaves open the question of the existence of an element $\theta$ such that $\theta_\edges>0$
for each $\pi > 0$. This is discussed in the next subsection. Item 2 shows that, while all $\pi$'s can be obtained if
loops are allowed ($\theta_V \ge 0$), only a convex subset of $\Delta(V)$ is obtained if we do not allow for loops
($\theta_V=0$).
\subsection{From a positive $\pi$  to positive transitions}
Given $\pi$, the fiber $\gamma^{-1}(\pi)$ is contained in an affine space parallel to the subspace
\begin{equation*}
  \ker(I_V+\frac12\Gamma) = \setof{\theta \in \reals^{V\cup\edges}}{\theta_V + \frac12 \Gamma \theta_{\edges} = 0}.
\end{equation*}
Each fiber $\gamma^{-1}(\pi)$, $\pi > 0$, contains special solutions. The solution $(\pi,0_{\edges})$ is not of interest because we want $\theta_\edges > 0$. If the graph has full connections, $\mathcal G = (V,V_2)$, there is the independence solution $\theta_{v} = \pi(v)^2$, $\theta_{\overline{vw}} = 2\pi(v)\pi(w)$.
e
If $\pi(v) > 0$, $v \in V$, a strictly positive solution is obtained as follows. Let $d(v)$ be the degree of the vertex $v$ and define a transition probability by $A(v,w)=1/(2d(v))$ if $\overline{vw}\in\edges$, $A(v,v)=1/2$, and $A(v,w)=0$ otherwise. $A$ is the transition matrix of a random walk on the graph $\mathcal G$, stopped with probability 1/2. Define a probability on $V\times V$ with $Q(v,w)=\pi(v)A(v,w)$. If $Q(v,w)=Q(w,v)$, $v,w\in V$, we are done: we have found a 2-reversible probability with marginal $\pi$ and such that $Q(v,w)>0$ if, and only if $\overline{vw} \in \edges$. Otherwise, if $Q(v,w) \ne Q(w,v)$ for some $v,w\in V$, we turn to the following Hastings-Metropolis construction.

\begin{proposition}\label{prop:M-H}
Let $Q$ be a probability on $V\times V$ such that $Q(v,w)>0$ if, and only if, $\overline{vw} \in \edges$ or $v=w$. Write $\pi(v) = \sum_w Q(v,w)$. Given $f: ]0,1[\times]0,1[ \to
]0,1[$ a symmetric function such that $f(x,y) \le x \wedge y$, then
    \begin{align*}
      P(v,w) =
      \begin{cases}
        f(Q(v,w),Q(w,v)) & \text{if $v \ne w$,} \\
        \pi(v) - \sum_{w \colon w \ne v} P(v,w) & \text{if $v=w$.}
      \end{cases}
    \end{align*}
is a 2-reversible probability on $V\times V$ such that $\pi(v) = \sum_w P(v,w)$ and $P(v,w)>0$ if, and only if, $\overline{vw} \in \edges$.
\end{proposition}
\begin{proof}
For $\overline{vw} \in \edges$ we have $P(v,w)=P(w,v) > 0$, otherwise zero. As $P(v,w) \le Q(v,w)$, $v \ne w$, it follows
\begin{align*}
  P(v,v) &= \pi(v) - \sum_{w \colon w \ne v} P(v,w) \\
         &\ge \sum_w Q(v,w) - \sum_{w \colon w \ne v} Q(v,w) \\
         &= Q(v,v) > 0.
\end{align*}
We have $\sum_w P(v,w) = \pi(v)$ by construction and, in particular, $P(v,w)$ is a probability on $V\times V$.
\end{proof}

\begin{remark}
\begin{enumerate}
\item The proposition applies to
\begin{enumerate}
\item $f(x,y) = x \wedge y$. This is the standard Hastings choice.
\item $f(x,y) = xy/(x+y)$. This was suggested by Barker.
\item $f(x,y) = xy$. In fact, as $y < 1$, we have $xy < x$.
\end{enumerate}
\item Given a joint probability $P$, the corresponding parameters
\begin{equation*}
\theta_{\overline{vw}} = 2P(v,w) \quad \textrm{and} \quad \theta_{v}= P(v,v)
\end{equation*}
are strictly positive for $\overline{vw}\in \edges$ and $v\in V$, otherwise zero. We have shown the existence of a mapping from $\pi$ in the interior of $\Delta(V)$ to a vector of parameters $\theta$ in the interior of $\Delta(V \cup \edges)$.
\end{enumerate}
\end{remark}
\subsection{Parameterization of reversible Markov matrices}\label{sec:RMC}
An q-reversible Markov matrix  $P$ supported on a graph $\mathcal G$ is parameterised by its non-zero extradiagonal values $P_{v\to w}$, i.e. by the elements of $\bigtimes_{v\in V} S^\circ(\neigh(v))$, where $S^\circ$ denotes the open solid simplex. As $\mathcal G$ is connected, the invariant probability of the Markov matrix $P$ is unique, therefore the joint 2-distribution is uniquely defined. If moreover the Markov matrix is reversible, the joint 2-distribution is symmetric and the $\theta$ parameters are computed. Viceversa, given the $\theta$'s, the transition matrix is given by
\begin{equation}\label{eq:theta}
  P_{v\to w} =\frac {\probof{X_n = v,X_{n+1} = w}}{\probof{X_n = v}}= \frac {\theta_{\overline{vw}}}{2\theta_v + \sum_{z \in D(v)} \theta_{\overline{vz}}}.
\end{equation}
The mapping $\theta \mapsto (P_{v\to w}\colon (v\to w) \in \arcs)$ is a rational mapping and the number of degrees of freedom is $\# V + \# \edges -1$.

Denoting $2\theta_v + \sum_{z \in D(v)} \theta_{\overline{vz}}$ by $\kappa(v)$, from \eqref{eq:theta}, the detailed balance conditions follow
\begin{equation*}
  \kappa(v) P_{v\to w} = \kappa(w) P_{w\to v}
\end{equation*}
and $\sum_v \kappa(v)=1$.

\subsection{A reversible Markov matrix is an auto-adjoint operator}
\label{sec:adjoint}
The detailed balance condition $\pi(v) P_{v\to w} = \pi(w) P_{w\to v}$ is equivalent to $P$ being adjoint as a linear operator on $L^2(\pi)$
\begin{equation*}
  \scalarat \pi {Pf} g = \scalarat \pi f {Pg}, \quad f,g \in L^2(\pi),
\end{equation*}
where $\scalarat \pi {Pf} g = \sum_{v}\left(\sum_{w} P_{v\to w} f(w)\right) g(v) \pi(v)$.

As $L^2(\pi)$ is isomorphic to the canonical Euclidian space $\reals^V$ via the linear mapping
\begin{equation*}
  I \colon L^2(\pi) \ni f \mapsto \diag(\pi)^{1/2} f \in \reals^V,
\end{equation*}
the Markov matrix $P$ is mapped to the symmetric matrix
\begin{equation*}
  S = I \circ P \circ I^{-1} = \diag(\pi)^{1/2}P\diag(\pi)^{-1/2}.
\end{equation*}

This implies that a reversible Markov matrix is diagonalizable; in particular the left eigenvector $\pi$ and the right eigenvector 1 of $P$ are both mapped to the eigenvector $(\pi(v)^{1/2}\colon v\in V)$ of $S$.

Each element of the symmetric matrix $S$ is positive if, and only if, the corresponding element of $P$ is positive. Each reversible Markov matrix $P$ with invariant probability $\pi$ is parameterized by a unique symmetric matrix $S$.

Viceversa, let $s(v,w)=s(w,v) > 0$ be defined for $\overline{vw}\in\edges$. Extend $s$ to all couples $v \ne w$, $\overline{vw} \notin \edges$ by $s(v,w)=0$. If
\begin{equation}\label{eq:s-inequality}
  \sum_{w\colon w\ne v} s(v,w) \sqrt{\pi(w)} \le \sqrt{\pi(v)},\quad v\in V,
\end{equation}
we can define
\begin{equation*}
  S(v,v) =  \frac1{\sqrt{\pi(v)}}\left(\sqrt{\pi(v)} - \sum_{w\ne v} s(v,w) \sqrt{\pi(w)}\right) \ge 0
\end{equation*}
to get a symmetric non-negative matrix $S=[s(v,w)]$, $v,w\in V$. The matrix $P = \diag{\pi}^{-1/2} S \diag{\pi}^{1/2}$ has non-negative entries; is a Markov matrix because
\begin{equation*}
  \sum_{w} P_{v\to w} = \sum_{w} \pi(v)^{-1/2} s(v,w) \pi(w)^{1/2} = 1, \quad v\in V;
\end{equation*}
satisfies the detailed balance equations
\begin{equation*}
  \pi(v) P_{v \to w} = \pi(v)^{1/2} \pi(w)^{1/2} s(v,w) = \pi(w) P_{w\to v}.
\end{equation*}

We can rephrase the computations above as follows.
\begin{proposition}\label{prop:squarerootpi}
  The set of Markov matrices which have structure $\mathcal G$ and are reversible with invariant probability $\pi$ is parameterized by
  \begin{equation*}
    \begin{cases}
      P_{v \to w} = \pi(v)^{-1/2} \pi(w)^{1/2} s(v,w),& (v\to w)\in\arcs,\\
      P_{v\to v} = 1 - \sum_{w\in\neigh(v) P_{v\to w}}
    \end{cases}
  \end{equation*}
the polytope of all positive weight functions $s \colon \edges \to \reals_{>0}$ and (unnormalized) positive probability $\pi$ satisfying the inequalities \eqref{eq:s-inequality}.
\end{proposition}

\subsection{Kolmogorov's theorem}
\label{sec:k-theorem}

Let $\mathcal G = (V,\edges)$ be a connected graph. For each closed path $\omega$, that is a path on the graph such that the last vertex coincides with the first one, $\omega=v_0 v_1 \dots v_{n} v_0$, we denote by $r(\omega)$ the reversed path $r(\omega)=v_0 v_{n} \dots  v_1 v_0$. The Kolmogorov's characterization of reversibility based on closed paths is well known. However, we give here a variation of the proof by \cite{suomela:1979} as it is an introduction to the algebraic arguments in the next Section. The proof has been modified to allow null transitions. If $\gamma = v_0v_1\cdots v_{n-1}v$ is any path connecting $v_0$ to $v$ we write $P^\gamma$ to denote the product of transitions along $\gamma$, i.e. $P^\gamma = \prod_{i=1,n} P_{v_{i-1}\to v_i}$.

\begin{theorem}[Kolmogorov's theorem]
The Markov irreducible matrix $P$ is reversible if, and only if, for all closed path $\omega$
\begin{equation} \label{eq:rev-circ}
P^\omega = P^{r(\omega)}
\end{equation}
\end{theorem}
\begin{proof}
Assume that the process is reversible. By multiplying together all detailed balance equations
\begin{equation*}
  \kappa(v_i) P_{v_i\to v_{i+1}} = \kappa(v_{i+1}) P_{v_{i+1}\to v_i},\quad i=0,1,\dots n, v_{n+1}=v_0,
\end{equation*}
and clearing the $\kappa$'s we obtain \eqref{eq:rev-circ}.

Viceversa, assume that all closed path have property \eqref{eq:rev-circ}. Fix once for all a vertex $v_0$ and consider a generic path $\gamma$ from $v_0$ to $v$. First we prove that there exists a positive constant $\kappa(v)$, depending only on $v$, such that $P^{\gamma}=\kappa(v)P^{r(\gamma)}$.  In fact, for any other path
$\gamma'=v_0 v_1' \dots v_{n'}' v$ with the same endpoints $v_0$ and $v$, $\gamma r(\gamma')$ is a closed path. Denoting by $k'(v)$ the corresponding constant, Kolmogorov's condition implies $k(v)=k'(v)$.
Moreover, for any vertex $w$ connected with $v$, consider the path $\gamma w$ and the corresponding constant $k(w)$. We have:
\begin{eqnarray*}
P^{\gamma}P_{v\rightarrow w}P_{w\rightarrow v}&=&k(v)P^{r\gamma}P_{v\rightarrow w}P_{w\rightarrow
v}\\
P^{\gamma w}P_{w\rightarrow v}&=&k(v)P^{r (\gamma w)}P_{v\rightarrow w}\\
k(w)P^{r(\gamma w)}P_{w\rightarrow v}&=&k(v)P^{r (\gamma w)}P_{v\rightarrow w}
\end{eqnarray*}
i.e. the detailed balance condition on $w$ and $v$.
\end{proof}

In the next Section we will discuss the algebraic interpretation of Kolmogorov condition.

\section{Algebraic theory}\label{sec:algebra}
The present section is devoted to the algebraic structure implied by the Kolmogorov's theorem for reversible Markov chains. We refer mainly to the textbooks \cite{berge:1985} and \cite{bollobas:1998} for graph theory, and to the textbooks \cite{cox|little|oshea:1997} and \cite{kreuzer|robbiano:2000} for computational commutative algebra. The theory of toric ideals is treated in detail in \cite{sturmfels:1996} and \cite{bigatti|robbiano:2001}. General references for algebraic methods in Stochastics are e.g. \cite{drton|sturmfels|sullivan:2009}, \cite{gibilisco|riccomagno|rogantin|wynn:2010}. The relevance of Graver bases, see \cite{sturmfels:1996}, has been pointed out to us by Shmuel Onn has in view of the applications discussed in \cite{deloera|hemmecke|onn|weismantel:2008} and \cite{onn:20xx}.
\subsection{Kolmogorov's ideal}
We denote by $\mathcal G = (V,\edges)$ an undirected graph. We split each edge into two opposite arcs to get a connected directed graph (without loops) denoted by $\mathcal D = (V,\arcs)$. The arc going from vertex $v$ to vertex $w$ is denoted by $v\to w$ or $(v\to w)$. The graph $\mathcal D$ is such that $(v\to v) \notin \arcs$ and $(v\to w)\in\arcs$ if, and only if, $(w\to v)\in \arcs$. Because of our application to Markov chains, we want two arcs on each edge, as it was explained in Figure \ref{fig:2arcs}.

The \emph{reversed} arc is the image of the 1-to-1 function $r \colon \arcs\to \arcs$ defined by $r(v\to w) = (w\to v)$. A \emph{path} is a sequence of vertices $\omega=v_0v_1\cdots v_n$ such that $(v_{k-1}\to v_k) \in \arcs$, $k=1,\dots,n$. The reversed path is denoted by $r(\omega)=v_nv_{n-1}\cdots v_0$. Equivalently, a path is a sequence of inter-connected arcs $\omega=a_1\dots a_n$, $a_k=(v_{k-1}\to v_k)$, and $r(\omega)=r(a_n)\dots r(a_1)$.

A \emph{closed path} $\omega = v_0 v_1 \cdots v_{n-1}v_0$ is any path going from a vertex $v_0$ to itself; $r(\omega) = v_0 v_{n-1} \cdots v_1 v_0$ is the reversed closed path. In a closed path any vertex can be the initial and final vertex.  If we do not distinguish any initial vertex, the equivalence class of paths is called a \emph{circuit}. A closed path is \emph{elementary} if it has no proper sub-closed-path, i.e. if does not meet twice the same vertex except the initial one $v_0$. The circuit of an elementary closed path is a \emph{cycle}. We denote by $\cycles$ the set of cycles of $\mathcal D$.

Consider the commutative indeterminates $P=[P_{v\to w}]$, $(v\to w)\in \arcs$, and the polynomial ring $k[P_{v\to w}:(v\to w)\in \arcs]$, i.e. the set of all polynomials in the indeterminates $P$ and coefficients in the number field $k$.

For each path $\omega = a_1\cdots a_n$, $a_k\in\arcs$, $k=1,\dots,n$, we define the monomial term
\begin{equation*}
  \omega = a_1\cdots a_n \mapsto P^{\omega} = \prod_{k=1}^n P_{a_k}.
\end{equation*}
For each $a\in\arcs$, let $N_a(\omega)$ be the number of traversals of the arc $a$ by the path $\omega$. Hence,
\begin{equation*}
  P^\omega = \prod_{a\in\arcs} P_a^{N_a(\omega)}.
\end{equation*}
%

Note that $\omega \mapsto P^{\omega}$ is a representation of the non-commutative concatenation of arcs on the commutative product of indeterminates.  Two closed paths associated to the same circuit are mapped to the same monomial term because they have the same traversal counts. The monomial term of a cycle is square-free because no arc is traversed twice.

Figure \ref{fig:6} presents the 6 cycles in the running example of Figure \ref{fig:2arcs}.%
\begin{figure}
     \begin{center}
      \includegraphics[scale=.5]{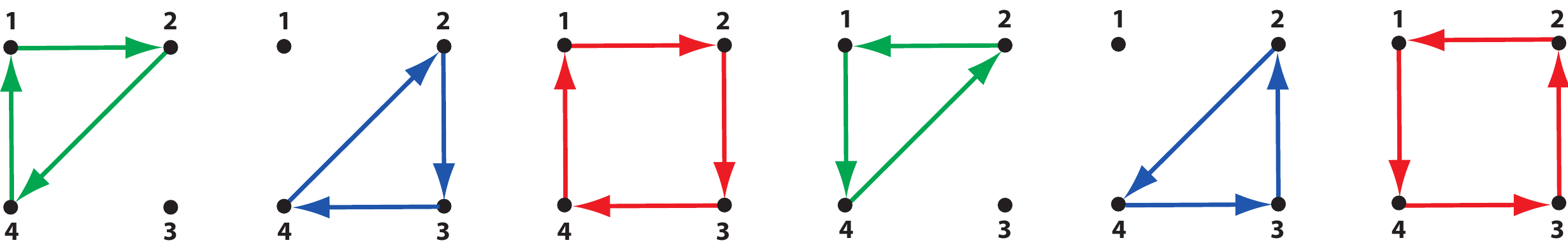}
    \end{center}
  \caption{The 6 cycles of a graph: $\omega_A=(1\to 2)(2\to 4)(4\to 1)$, $\omega_B = (2\to 3)(3\to 4)(4\to 2)$, $\omega_C = (1\to 2)(2\to 3)(3\to 2)(4\to 1)$, $\omega_D = r(\omega_A)$, $\omega_E=r(\omega_B)$,$\omega_F=r(\omega_C)$.\label{fig:6}}
\end{figure}
This list of cycles is larger than a basis of cycles in the undirected graph $\mathcal G$, for instance $\set{\omega_A,\omega_B}$. We will see below that all directed cycles are needed for the algebraic argument.

\begin{definition}[K-ideal]
The \emph{Kolmogorov's ideal} or \emph{K-ideal} of the graph $\mathcal G$ is the ideal of the ring $k[P_{v\to w}:(v\to
w)\in \arcs]$ generated by the binomials $P^\omega-P^{r(\omega)}$, where $\omega$ is \emph{any circuit}. The K-variety is
the $k$-affine variety of the K-ideal.
\end{definition}
 Our main application concerns the real case $k=\reals$, but the combinatorial structure of the K-ideal does not depend on the choice of a specific field.  A interesting choice for computations could be the Galois field $k = \mathbb Z_2$.

\begin{proposition}[Examples of K-ideals]
Let the Markov matrix $P$ with structure $\mathcal G$ be reversible.
\begin{enumerate}
\item
The real vector $P_{v\to w}$,  $(v\to w) \in \arcs$, is a point of the intersection of the variety of the K-ideal with $S(\mathcal D) = \bigtimes_{v \in V} S(v)$, where
\begin{equation*}
  S(v) = \setof{P_{a}\in \reals_+^{\out(v)}}{\sum_{a\in\out(v)} P_{a}(w) \le 1}.
\end{equation*}
\item
Let $(X_n)_{n\ge 0}$ be the stationary Markov chain with transition $P$. Then the real vector of joint probabilities
$p(v,w)=\probof{X_n=u,X_{n+1}=v}$, $(v\to w)\in \arcs$, is a point in the intersection of the K-variety and the simplex
\begin{equation*}
  S(\arcs)=\setof{p\in \reals_+^\arcs}{\sum_{a\in\arcs} P(a) \le 1}.
\end{equation*}
\end{enumerate}
\end{proposition}
\begin{proof}
  \begin{enumerate}
  \item It is the first part of the Kolmogorov's theorem.
  \item
Let $\omega=v_0\dots v_n v_0$ be a closed path. If $\pi$ is the stationary probability, by multiplying the Kolmogorov's
equations by the product of the initial probabilities at each transition, we obtain
    \begin{equation*}
      \pi(v_0) \pi(v_1)\cdots \pi(v_n) P_{v_0\to v_1} \cdots P_{v_n\to v_0} = \pi(v_0) \pi(v_n)\cdots \pi(v_n) P_{v_0\to v_n} \cdots P_{v_1\to v_0},
    \end{equation*}
hence
\begin{equation*}
p(v_0,v_1) p(v_1,v_2) \cdots p(v_n,v_0) = p(v_0,v_n)p(v_n,v_{n-1})\cdots p(v_1,v_0).
\end{equation*}
  \end{enumerate}
However, in this case the Kolmogorov's equations are trivially satisfied as $p(v,w)=p(w,v)$.
\end{proof}

The K-ideal has a finite basis because of the Hilbert's basis theorem. Precisely, a finite basis is obtained by restricting to cycles, which are finite in number. We underline that here we consider all the cycles, not just a generating set of cycles. The related result by \cite[Th. 1]{mitrophanov:2004} is discussed in the next subsection.

\begin{proposition}[Cycle basis of the K-ideal]\label{prop:cyclebasis}
The K-ideal is generated by the set of binomials $P^\omega - P^{r(\omega)}$, where $\omega$ is cycle.
\end{proposition}

\begin{proof}
Let $\omega=v_0v_1\cdots v_0$ be a closed path which is not elementary and consider the least $k\ge1$ such that $v_k=v_{k'}$ for some $k'<k$. Then the sub-path $\omega_1$ between the $k'$-th vertex and the $k$-th vertex is an elementary closed path and the residual path $\omega_2=v_0\cdots v_{k'}v_{k+1}\cdots v_0$ is closed and shorter than the original one. The arcs of $\omega$ are in 1-to-1 correspondence with the arcs of $\omega_1$ and $\omega_2$, hence $N_a(\omega) = N_a(\omega_1)+N_a(\omega_2)$, $a\in\arcs$. The procedure can be iterated and stops in a finite number of steps.  Hence, given any closed path $\omega$, there exists a finite sequence of cycles $\omega_1,\dots,\omega_l$, such that the list of arcs in $\omega$ is partitioned into the lists of arcs of the $\omega_i$'s. From $P^{\omega_i}-P^{r(\omega_i)} = 0$, $i=1,\dots,l$, it follows
\begin{equation*}
 P^\omega = \prod_{i=1}^l P^{\omega_i} = \prod_{i=1}^l P^{r(\omega_i)} = P^{r(\omega)}.
\end{equation*}

\end{proof}
The K-ideal is generated by a finite set of binomials and this set has the same number of elements as the set of undirected cycles of $\mathcal G$. If the graph is a tree, the cycles reduce to those of the type $v_1\to v_2\to v_1$ and the binomials are identities, hence the K-ideal is equal the the full ring. This is an involved way to prove all stationary Markov process on a tree is reversible \cite[Lemma 1.5]{kelly:1979}.

The cycle basis of Proposition  \ref{prop:cyclebasis} belongs to the special class of bases, namely Gr\"obner bases. We refer to the textbooks \cite{cox|little|oshea:1997} and \cite{kreuzer|robbiano:2000} for a detailed
discussion. We review the basic definitions of this theory, which is based on the existence of a monomial order $\succ$, i.e. a total order on monomial terms which is compatible with the product. Given such an order, the leading term $\LT(f)$ of the polynomial $f$ is defined. A generating set is a \emph{Gr\"obner basis} if the set of leading terms of the ideal is generated by the leading terms of monomials of the generating set. A Gr\"obner basis is \emph{reduced} if the coefficient of the leading term of each element of the basis is 1 and no monomial in any element of the basis is in the ideal generated by the leading terms of the other element of the basis. The Gr\"obner basis property depends on the monomial order. However, a generating set is said to be a \emph{universal} Gr\"obner basis if it is a Gr\"obner basis for all monomial orders.

The finite algorithm for testing the Gr\"obner basis property depends on the definition of \emph{syzygy}.  The syzygy of two polynomial $f$ and $g$ is the polynomial
\begin{equation*}
  \syz(f,g) = \frac{\LT(g)}{\gcd(\LT(f),\LT(g))}f - \frac{\LT(f)}{\gcd(\LT(f),\LT(g))}g.
\end{equation*}
A generating set of an ideal is a Gr\"obner basis if, and only if, it contains the syzygy $\syz(f,g)$ whenever it contains the polynomials $f$ and $g$, see \cite[Ch 6]{cox|little|oshea:1997} or \cite[Th. 2.4.1 p. 111]{kreuzer|robbiano:2000}.

\begin{proposition}[Universal G-basis]\label{prop:cycles}
The cycle basis of the K-ideal is a\marginpar{the?} reduced universal Gr\"obner basis.
\end{proposition}
\begin{proof}
Choose any monomial order $\succ$ and let $\omega_1$ and $\omega_2$ be two cycles with $\omega_i \succ r(\omega_i)$, $i=1,2$. Assume first they do not have any arc in common. In such a case $\gcd(P^{\omega_1},P^{\omega_2})=1$ and the syzygy is
\begin{multline*}
  \syz(P^{\omega_1}-P^{r(\omega_1)},P^{\omega_2}-P^{r(\omega_2)}) = \\ P^{\omega_2}(P^{\omega_1}-P^{r(\omega_1) }) - P^{\omega_1}(P^{\omega_2}-P^{r(\omega_2)}) = P^{\omega_1}P^{r(\omega_2)} - P^{r(\omega_1)}P^{\omega_2},
\end{multline*}
which belongs to the K-ideal.

Let now $\alpha$ be the common part, that is $\gcd(P^{\omega_1},P^{\omega_2}) = P^\alpha$. The syzygy of $P^{\omega_1} - P^{r(\omega_1)}$ and $P^{\omega_2} - P^{r(\omega_2)}$ is
\begin{equation*} P^{\omega_1-\alpha} P^{r(\omega_2)} - P^{\omega_2-\alpha} P^{r(\omega_1)} = P^{r\alpha}(P^{\omega_1-\alpha} P^{r(\omega_2)-r\alpha} - P^{\omega_2-\alpha} P^{r(\omega_1)-r\alpha}), \end{equation*}
which again belongs to the K-ideal because $\omega_1-\alpha+r(\omega_2-\alpha)$ is a cycle. In fact $\omega_1-\alpha$ and $\omega_2-\alpha$ have in common the extreme vertices, corresponding to the extreme vertices of $\alpha$. Notice that $\alpha$ is the common part of $\omega_1$ and $\omega_2$ only if it is traversed in the same direction by the both cycle. The previous proof does not depend on the choice of the leading term of the binomials, therefore the Gr\"obner basis is universal. The Gr\"obner basis is reduced because no monomial of a cycle can divide a monomial of a different cycle.
\end{proof}
\begin{example}[Running example continue]
Figure \ref{fig:1} is an illustration of the proof.
\begin{figure}
  \begin{center}
    \begin{tabular}{c}
\includegraphics[scale=.5]{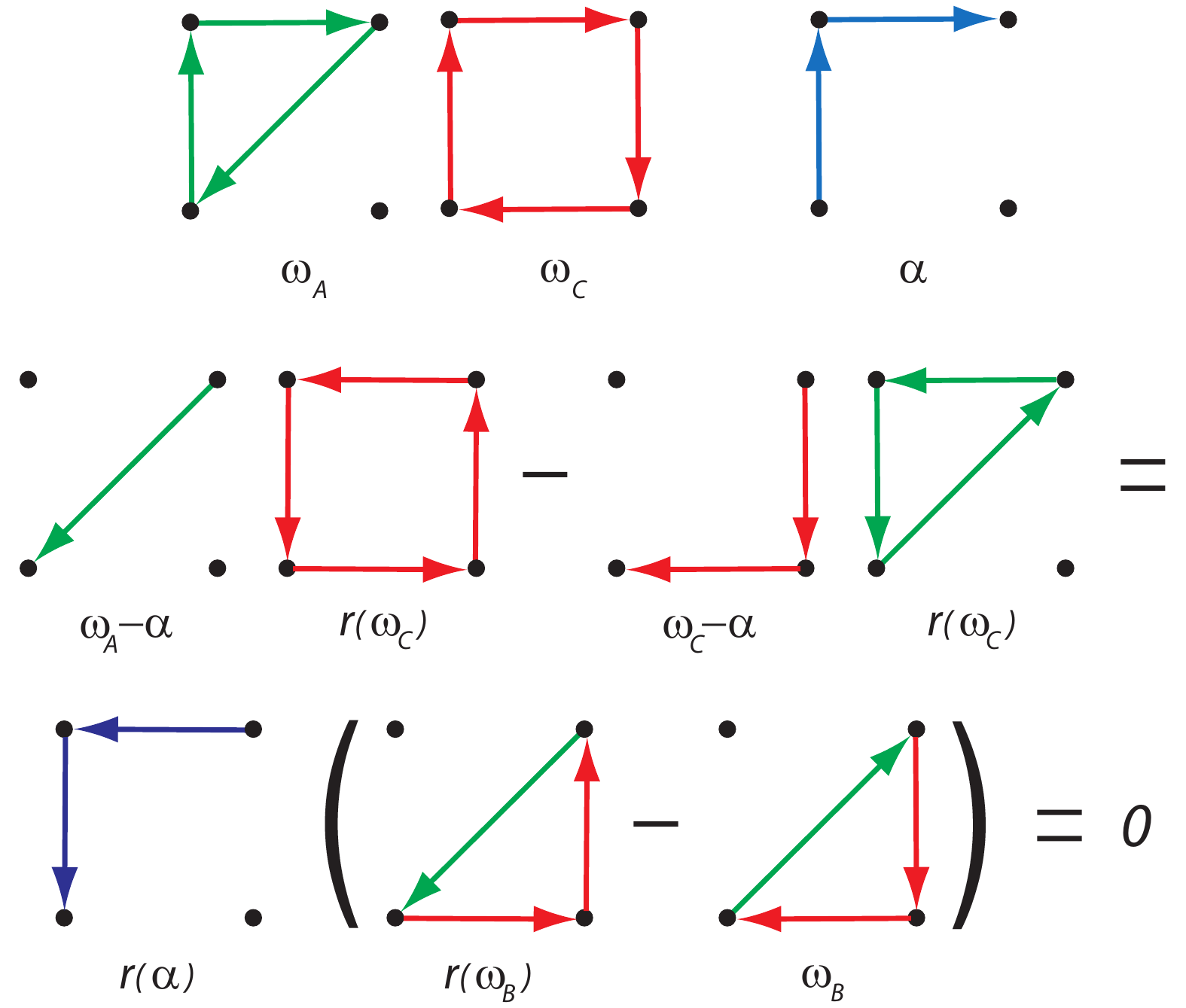}
\end{tabular}
  \end{center}
  \caption{Running example: illustration of the proof of Proposition \ref{prop:cycles} \label{fig:1}}
\end{figure}

On this example we can see why the two $\omega_A$ and $\omega_C$ do not generate the K-ideal. In fact, from
\begin{align*}
  P_{1\to 2} P_{2\to 4} P_{4\to 1} &= P_{1\to 4} P_{4\to 2} P_{2\to 1} \\
  P_{1\to 2} P_{2\to 3} P_{3\to 4} P_{4\to 1} &= P_{1\to 4} P_{4\to 3} P_{3\to 2} P_{2\to 1}
\end{align*}
it follows
\begin{equation*}
  P_{4\to 1} P_{1\to 2} P_{2\to 1} P_{1\to 4}\left(P_{2\to 3} P_{3\to 4} p_{4\to 2} -  P_{2\to 4} P_{4\to 3} P_{3\to 2} \right) = 0
\end{equation*}
which, in turn, gives the binomial of $\omega_B$ if $P_{4\to 1} P_{1\to 2} P_{2\to 1} P_{1\to 4} \ne 0$ and, therefore, the factor $P^\alpha P^{r(\alpha)}$ can be cleared. This is confirmed by the use of a symbolic algebraic software such as \cocoa, see \cite{CocoaSystem}. This computation shows that the $\omega_B$ equation does not belong to the ideal generated by the $\omega_A, \omega_C$ equations unless we add the condition $P_{4\to 1} P_{1\to 2} P_{2\to 1} P_{1\to 4} \ne 0$. Notice that $\omega_A$ and $\omega_C$ are the cycles obtained from the spanning tree $3\to 4 \to 1\to 2$.
\end{example}

\begin{example}[Running example: Monomial basis of the quotient ring]
Take any order on vertexes, e.g. $1 \prec 2 \prec 3 \prec 4$, and derive a lexicographic order on arcs: $1\to 2\prec 1\to 4\prec 2\to 1\prec 2\to 3\prec 2\to 4\prec 3\to 2\prec 3\to 4\prec 4\to 1\prec 4\to 2\prec 4\to 3$. Take the same order on indeterminates $P_a$, $a\in\arcs$, and the lexicographic order on monomials. We check that the leading terms of the there binomials in the G-basis are $P^{r(\omega_A)}, P^{r(\omega_B)}, P^{r(\omega_C)}$, see Figure \ref{fig:6}. The exponents of the leading terms of the G-basis are
\begin{equation*}\footnotesize\setlength{\arraycolsep}{3pt}\begin{array}{lcccccccccc}
\arcs &1\to 2 & 1\to 4 & 2\to 1 & 2\to 3 & 2\to 4 & 3\to 2 & 3\to 4 & 4\to 1 & 4\to 2 & 4\to 3 \\ \hline
N(r(\omega_A))&0&1&1&0&0&0&0&0&1&0 \\
N(r(\omega_B))&0&0&0&0&1&1&0&0&0&1 \\
N(r(\omega_C))&0&1&1&0&0&1&0&0&0&1
   \end{array}
\end{equation*}
Each monomial $P^{N}= \prod_{a\in\arcs} P_a^{N_a}$, $N\in \integers_{\ge}^\arcs$, is reduced by the K-ideal to a monomial whose exponent does not contain any of the counts in the table. E.g. $P_{1\to 4}P_{2\to 1}P_{3\to 2}$ is an element of the monomial basis of the polynomial ring $\mod$ the K-ideal.
\end{example}

\subsection{Cycle and cocycle spaces}
We adapt to our context some standard tools of algebraic graph theory, namely the cycle an cocycle spaces, see e.g. \cite[Ch 2]{berge:1985} and \cite[II.3]{bollobas:1998}.

Let $\mathcal C$ be the set of cycles.  For each cycle $\omega \in \mathcal C$ we define the \emph{cycle vector} of $\omega$ to be $z(\omega) = (z_a(\omega): a\in\arcs)$, where
\begin{equation*}
  z_a(\omega) = \begin{cases}
+1 & \text{if $a$ is an arc of $\omega$}, \\ -1 & \text{if $r(a)$ is an arc of $\omega$},\\
 0 & \text{otherwise.}
  \end{cases}
\end{equation*}
Note that $z_{r(a)}(\omega)=-z_a(\omega)$.  If $z^+$ and $z^-$ are the positive and the negative part of $z$, respectively, then $z^+_a(\omega)=N_a(\omega)$ and $z^-_a(\omega)=N_a(r(\omega))$. It follows that $P^\omega = P^{N(\omega)} = P^{z^+(\omega)} = \prod_{a\in\arcs} P_a^{z_a^+(\omega)}$ and
\begin{equation}
  \label{eq:omegavscycle}
P^\omega-P^{r(\omega)} = P^{z^+(\omega)}-P^{z^-(\omega)}.
\end{equation}
More generally, the definition can be is extended to any circuit $\omega$
by defining
\begin{equation*}
  z_a(\omega) = N_a(\omega)-N_{r(a)}(\omega).
\end{equation*}
The equality $z^+(\omega) = N(\omega)$ holds if, and only if, $a\in\omega$ implies $r(a) \notin\omega$, $a\in\arcs$.

Let $Z(\mathcal D)$ be the \emph{cycle space}, i.e. the vector space generated in $\reals^{\arcs}$ by the cycle vectors.

For each proper subset $B$ of the set of vertices, $\emptyset \ne B \subsetneq V$ we define the \emph{cocycle vector} of $B$ to be $u(B) = (u_a(B): a \in \arcs)$, with
\begin{equation*}
  u_a(B) = \begin{cases}
+1 & \text{if $a$ exits from $B$,} \\
-1 & \text{if $a$ enters into $B$,} \\
0 & \text{otherwise.}
  \end{cases} \qquad  a\in\arcs
\end{equation*}
See an example in Figure \ref{fig:cocycle}.
\begin{figure}
  \centering
  \includegraphics[scale=.5]{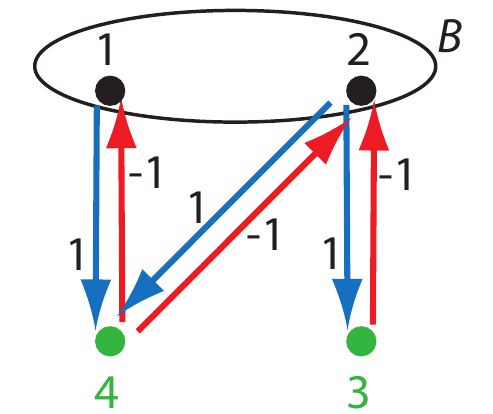}
  \caption{Example of cocycle vector. All arcs not shown take value 0. \label{fig:cocycle}}
\end{figure}
Note that $u_{r(a)}(B)=-u_a(B)$.

Let $U(\mathcal D)$ be the \emph{cocycle space}, i.e. the vector space generated in $\reals^{\arcs}$ by the cocycle vectors. Let $U$ be the matrix whose rows are the cocycle vectors $u(B)$, $\emptyset \ne B \subsetneq V$.  The matrix $U=[u_a(B)]_{\emptyset \ne B\subsetneq V,a\in \arcs}$ is the \emph{cocycle matrix}.

The cycle space and the cocycle space are orthogonal in $\reals^\arcs$. In fact, for each cycle vector $z(\omega)$ and cocycle vector $u(B)$, we have
\begin{equation*}
  z_{r(a)}(\omega) u_{r(a)}(B) = (-z_a(\omega))(-u_a(B)) = z_a(\omega)u_a(B), \quad a\in\arcs,
\end{equation*}
so that
\begin{align*}
  z(\omega)\cdot u(B) &= \sum_{a\in\arcs} z_a(\omega) u_a(B) = \sum_{a\in\omega} z_a(\omega) u_a(B) + \sum_{r(a)\in\omega} z_a(\omega) u_a(B) \\ &= 2 \sum_{a\in\omega} z_a(\omega) u_a(B) = 2 \left[\sum_{a\in\omega,u_a(B) = +1} 1 - \sum_{a\in\omega,u_a(B) = -1} 1\right] = 0.
\end{align*}
It is shown e.g. in the previous references that the cycle space is the orthogonal complement of the cocycle space for undirected graphs. In our setting it is the orthogonal complement relative to the subspace of vectors $x$ such that $x_{r(a)}=-x_a$. As we are interested in elements of the cycle space with integer entries, i.e. those elements $z=(z_{a}\colon a \in \arcs)$ of the cycle space that can be exponents of a monomial $P^z = \prod_{a\in\arcs} P_a^{z_a}$, we are going to use the following matrix encoding of our problem.
\begin{definition}[Model matrix]\label{def:modelmatrix}
Consider the matrix $E = [E_{e,a}]_{e\in\edges, a\in\arcs}$ whose element $E_{e,a}$ in position $(e,a)$ is 1 if the arc $a$ is une of the directions of the edge $e$, zero otherwise. Let $U$ be the cocycle matrix. The \emph{model martix} is the block matrix
\begin{equation*}
  A=\begin{bmatrix}E \\ U\end{bmatrix}.
\end{equation*}
It follows $Z(\mathcal D) = \ker A$.  A \emph{lattice basis} of the lattice $\integers(\mathcal D) = \ker A \cap \integers^\arcs$ is a linear basis of $\ker A$ with integer entries.
\end{definition}

The matrix $A$ has dimension $ \# \mathcal E +\#V -1$. In fact $E$ can be re-arranged as $[I_{\# \mathcal E} \vert I_{\# \mathcal E}]$, with $I_{\# \mathcal E}$ the identity matrix, and $U$ has $\# V -1$ linearly independent rows, the dimension of the cocycle space.

\begin{table}
\caption{Running example. Model matrix $A$ and a basis of $\integers(\mathcal D)$.\label{tab:A}}
{
\begin{equation*}
\begin{array}{c|rrrrr|rrrrr|} \multicolumn{1}{c}{}& 1\to 2 & 1\to 4 & 2\to 3 &
2\to 4 &\multicolumn{1}{c}{3\to 4} &
 2\to 1 & 4\to 1 & 3\to 2 & 4\to 2 & \multicolumn{1}{c}{4\to 3}\\
\cline{2-11}
 \overline{12} & 1 & 0 & 0 & 0 & 0 & 1 & 0 & 0 & 0 & 0 \\
 \overline{14} & 0 & 1 & 0 & 0 & 0 & 0 & 1 & 0 & 0 & 0 \\
 \overline{23} & 0 & 0 & 1 & 0 & 0 & 0 & 0 & 1 & 0 & 0 \\
 \overline{24} & 0 & 0 & 0 & 1 & 0 & 0 & 0 & 0 & 1 & 0 \\
 \overline{34} & 0 & 0 & 0 & 0 & 1 & 0 & 0 & 0 & 0 & 1 \\ \cline{2-11}
{\mathbf{ \{1\}}} &  \mathbf{1} &  \mathbf{1} &  \mathbf{0} &  \mathbf{0} &  \mathbf{0} &  \mathbf{-1} &  \mathbf{-1}
&  \mathbf{0} &  \mathbf{0} &  \mathbf{0} \\
 \{2\} & -1 & 0 & 1 & 1 & 0 & 1 & 0 & -1 & -1 & 0 \\
\mathbf{ \{3\}} &  \mathbf{0} &  \mathbf{0} &  \mathbf{-1} &  \mathbf{0} &  \mathbf{1} &  \mathbf{0} &  \mathbf{0}
&  \mathbf{1} &  \mathbf{0} &  \mathbf{-1} \\
\{4\} & 0 & -1 & 0 & -1 & -1 & 0 & 1 & 0 & 1 & 1 \\
  \mathbf{\{12\}} &  \mathbf{0} &  \mathbf{1} &  \mathbf{1} &  \mathbf{1} &  \mathbf{0} &  \mathbf{0} &  \mathbf{-1}
  &  \mathbf{-1} &  \mathbf{-1} &  \mathbf{0} \\
 \{13\} & 1 & 1 & -1 & 0 & 1 & -1 & -1 & 1 & 0 & -1 \\
 \{14\} & 1 & 0 & 0 & -1 & -1 & -1 & 0 & 0 & 1 & 1 \\
 \{23\} & -1 & 0 & 0 & 1 & 1 & 1 & 0 & 0 & -1 & -1 \\
 \{24\} & -1 & -1 & 1 & 0 & -1 & 1 & 1 & -1 & 0 & 1 \\
 \{34\} & 0 & -1 & -1 & -1 & 0 & 0 & 1 & 1 & 1 & 0 \\
 \{123\} & 0 & 1 & 0 & 1 & 1 & 0 & -1 & 0 & -1 & -1 \\
 \{124\} & 0 & 0 & 1 & 0 & -1 & 0 & 0 & -1 & 0 & 1 \\
 \{134\} & 1 & 0 & -1 & -1 & 0 & -1 & 0 & 1 & 1 & 0 \\
 \{234\} & -1 & -1 & 0 & 0 & 0 & 1 & 1 & 0 & 0 & 0\\
 \cline{2-11}
\multicolumn{11}{c}{}\\
 \cline{2-11}
 z(\omega_A) &1& -1& 0& 1& 0& -1& 1& 0& -1& 0\\
z(\omega_B)  &0& 0& 1& -1& 1& 0& 0& -1& 1& -1\\
 \cline{2-11}
\end{array}
\end{equation*}
}
\end{table}
\begin{example}[Running example continue]
Table \ref{tab:A} shows the matrix $A$ and a lattice basis of $\integers(\mathcal D)$ computed with \cocoa.  The top matrix is the $E$ matrix; the bottom matrix is the $U$ matrix, where three linearly independent rows are highlighted. The two bottom row vectors are the lattice basis.
\end{example}

\subsection{Toric ideal}
We want to show that the K-ideal is the toric ideal of the model matrix $A$, see Definition \ref{def:modelmatrix}. Basic definitions and theory are in \cite[Ch 4]{sturmfels:1996}, see also \cite{bigatti|robbiano:2001}.

Consider the polynomial ring $\rationals[P_a \colon a\in \arcs]$ and  the Laurent polynomial ring $\rationals[t_e^{\pm 1}, t_B^{\pm 1} \colon e\in\edges, \emptyset \ne B\subsetneq V]$, together with their homomorphism $h$ defined by
  \begin{equation}\label{eq:A-model}
    h \colon P_a \longmapsto \prod_e t_e^{E_{e,a}}\prod_{B} t_B^{u_a(B)} = t^{A(a)}.  \end{equation}
As $E_{e,a} = E_{e,r(a)}$ for all edge $e$ and arc $a$, the first factor in \eqref{eq:A-model} is a symmetric function $s(v,w) = \prod_e t_e^{E_{e,v\to w}} = s(w,v)$. We could write
\begin{equation}\label{eq:h-homo}
  h \colon P_{v\to w} \longmapsto s(v,w) \prod_{B} t_B^{u_{v\to w}(B)}.
\end{equation}

The kernel $I(A)$ of $h$ is called the \emph{toric ideal} of $A$,
\begin{equation*}
 I(A) = \setof{f\in \rationals[P_a:a\in \arcs]}{f(t^{A(a)} \colon a \in \arcs)=0}.
\end{equation*}
The toric ideal $I(A)$ is a prime ideal and the binomials
\begin{equation*}
  P^{z^+} - P^{z^-}, \quad z \in \integers^\arcs, \quad Az=0,
\end{equation*}
are a generating set of $I(A)$ as a $\rationals$-vector space.  A finite generating set of the ideal is formed by selecting a finite subset of such binomials. The basis we find is a Graver basis.

We recall the definition of Graver basis as it is given in \cite{deloera|hemmecke|onn|weismantel:2008} and we apply it to the \emph{cycle lattice} $\integers(\mathcal D)$. We introduce a partial order and its set of minimal elements as follows.
\begin{definition}[Graver basis]\
Let $z_1$ and $z_2$ be two element of the cycle lattice $\integers(\mathcal D)$.
\begin{enumerate}
\item $z_1$ is \emph{conformal} to $z_2$, $z_1 \sqsubseteq z_2$, if the component-wise product is non-negative (i.e. $z_1$ and $z_2$ are in the same quadrant) and $|z_1| \le |z_2|$ component-wise, i.e. $z_{1,a}z_{2,a}\ge 0$ and $z_{1,a} \le z_{2,a}$ for all $a \in \arcs$.
\item A \emph{Graver basis} of $\integers(\mathcal D)$ is the set of the minimal elements with respect to the conformity partial order $\sqsubseteq$.
  \end{enumerate}
\end{definition}
\begin{proposition}\label{pr:graver}\
\begin{enumerate}
\item\label{item:graver1}
For each cycle vector  $z \in \integers(\mathcal D)$, $z = \sum_{\omega \in \mathcal C} \lambda(\omega) z(\omega)$, $\lambda(\omega) \in \rationals$, there exist cycles $\omega_1, \dots, \omega_n \in \mathcal C$ and positive integers $\alpha(\omega_1),\dots,\alpha(\omega_n)$, such that $z^+\ge z^+(\omega_i)$, $z^-\ge z^-(\omega_i)$, $i=1,\dots,n$, and
\begin{equation*}
  z = \sum_{i=1}^n \alpha(\omega_i) z(\omega_i) .
\end{equation*}
\item
The set $\left\{z(\omega) \colon \omega \in \mathcal C \right\}$  is a \emph{Graver basis} of  $\integers (\mathcal D)$.
\end{enumerate}

\end{proposition}
\begin{proof}
\begin{enumerate}
\item
For all $\omega\in\mathcal C$ we have $-z(\omega)=z(r(\omega))$, so that we can assume all the $\lambda(\omega)$'s to be non-negative. Notice also that we can arrange things in such a way that at most one of the two direction of each cycle has a non-zero coefficient. We define
  \begin{equation*}
    \arcs_+(z)=\setof{a\in\arcs}{z_a>0}, \quad \arcs_-(z)=\setof{a\in\arcs}{z_a<0},
  \end{equation*}
  and consider two subgraph of $\mathcal D$ with a the set of arcs restricted to $\arcs_+$, $\arcs_-$, respectively. We note that $r\arcs_+(z)=\arcs_-(z)$ and $r\arcs_-(z)=\arcs_+(z)$; in particular, both $\arcs_+(z)$, $\arcs_-(z)$ are not empty. We drop from now on the dependence on $z$ for ease of notation.

We show first there is a cycle whose arcs are in $\arcs_+$. If not, if a cycle of full graph $\mathcal D$ has one arc in $\arcs_+$, it would exists vertex $v$ such that $\out(v)\cap\arcs_+=\emptyset$ while $\inv(v)\cap\arcs_+\ne\emptyset$. Let $u(v)$ be the cocycle vector of $\set{v}$; we derive a contradiction to the assumption $z\cdot u(v) = 0$. In fact,
\begin{multline*}
  z \cdot u(v) = \sum_{a\in\arcs_+} z_a u_a(v) + \sum_{a\in\arcs_-} z_a u_a(v) \\ = \sum_{a\in\arcs_+} z_a u_a(v) + \sum_{a\in\arcs_+} z_{r(a)} u_{r(a)}(v) =
  2\sum_{a\in\arcs_+} z_a u_a(v) \ne 0
  \end{multline*}
because each of the terms $z_au_a(v)$, $a\in\arcs_+$, is either 0 or equal to $-z_a < 0$ if $a\in\inv(v)$.

Let $\omega_1$  be a cycle in $\arcs_+$ and define an integer $\alpha(\omega_1) \ge 1$ such that $z^+ -
\alpha(\omega_1) z^+(\omega_1) \ge 0$ and it is zero for at least one $a$. The vector $z^1 = z - \alpha(\omega_1)
z(\omega_1)$ belongs to the cycle space $\integers(\mathcal D)$, and moreover $\arcs_+(z^1) \subset \arcs_+(z)$.

By  repeating the same step a finite number of times we obtain a new representation of $z$ in the form $z =
\sum_{i=1}^n \alpha(\omega_i) z(\omega_i)$ where the support of each $\alpha(\omega_i) z^+(\omega_i)$ is contained in
$\arcs_+$. It follows
\begin{equation}\label{eq:z-conf}
z^+ = \sum_{i=1}^n  \alpha(\omega_i) z^+(\omega_i)\quad \textrm{and} \quad z^- = \sum_{i=1}^n \alpha(\omega_i)
z^-(\omega_i)
\end{equation}
\item
In the previous decomposition each $z(\omega_i)$, $i=1,\dots,n$, is conformal to $z$.  In fact, from $z^+\ge z^+(\omega_i)$ and $z^-\ge
z^-(\omega_i)$, it follows  $z_a z_a(\omega_i)=z^+_a z^+_a(\omega_i)-z^-_a z^-_a(\omega_i) \ge 0$ and $\vert z_a(\omega_i)\vert= z^+_a(\omega_i)- z^-_a(\omega_i)\le z_a^+ + z_z^-=\vert z_a \vert$. Therefore $z(\omega_i) \sqsubseteq z$.
\end{enumerate}

\end{proof}

\begin{example}[Running example continue]
We give an illustration of the previous proof. Consider the cycle vectors
\begin{equation*}
\setlength{\arraycolsep}{1.8pt}
\renewcommand{\arraystretch}{1.1}
\begin{array}{c c c c c c c c c c c c}
            & 1\to2&2\to1&2\to3&3\to2&3\to4&4\to3&4\to1&1\to4&2\to4&4\to2& \\
            \cline{2-11}
z(\boldsymbol{\textcolor{OliveGreen}{\omega_A}})= (& 1        &  -1     &  0      &  0      &  0      &  0      &  1      &  -1      &  1      & -1 & ) \\
z(\boldsymbol{\textcolor{blue}{\omega_B}})= (& 0        &   0     &  1      &  -1     &  1      & -1      &  0      &  0      &  -1     &  1  & )\\
z(\boldsymbol{\textcolor{red}{\omega_C}})= (& 1        &  -1     &  1      &  -1     &  1      &  -1     &  1      &  -1      &  0      &  0& )
\end{array}
\end{equation*}
and the element of the cycle space $z=z(\omega_A)+2z(\omega_B)+2z(\omega_C)$, see Figure \ref{fig:zeta-omega}.
\begin{figure}
  \begin{center}
  \includegraphics[width=9cm]{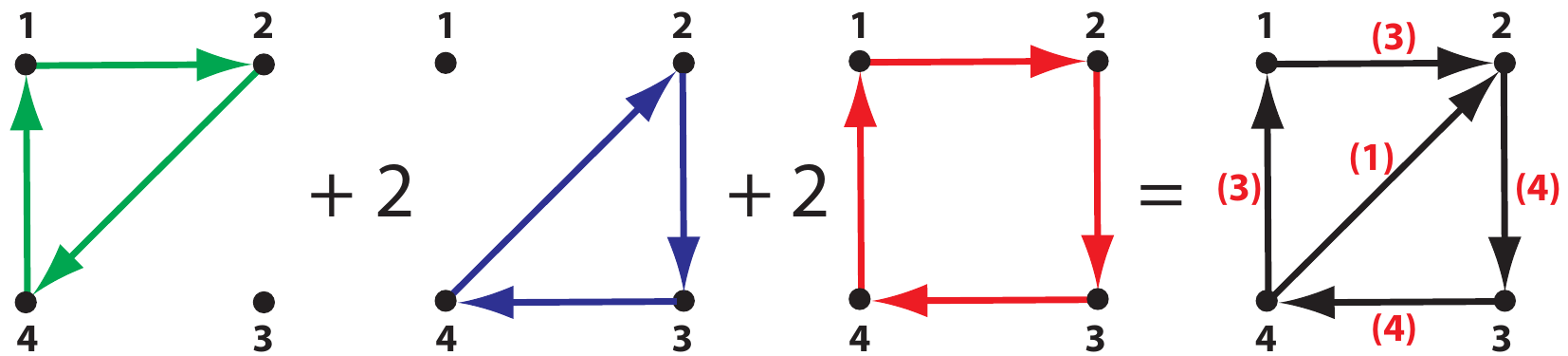}
\end{center}
  \caption{Running example. Cycle space.\label{fig:zeta-omega}}
\end{figure}
We have
\begin{equation*}
\setlength{\arraycolsep}{0.8pt}
\renewcommand{\arraystretch}{1.1}
\begin{array}{r c c r c l c l c l c l c l c l c l c l r}
z =& z(\boldsymbol{\textcolor{OliveGreen}{\omega_A}})+ 2 z(\boldsymbol{\textcolor{blue}{\omega_B}})+
2 z(\boldsymbol{\textcolor{red}{\omega_C}}) & =  (3&,&-3&,&4&,&-4&,&4&,&-4&,&0&,&0&,&  -1 &,& 1 )\\
z^+ =&  z^+(\boldsymbol{\textcolor{blue}{\omega_B}})+3z^+(\boldsymbol{\textcolor{red}{\omega_C}}) & =  (3 & , &0 &,& 4  &,& 0 &,& 4  &,& 0  & ,& 0  &,&0 &,& 0 &,& 1 )\\
\end{array}
\end{equation*}
as it is illustrated in Figure  \ref{fig:zeta-omega2}.
\begin{figure}
  \begin{center}
  \includegraphics[width=9cm]{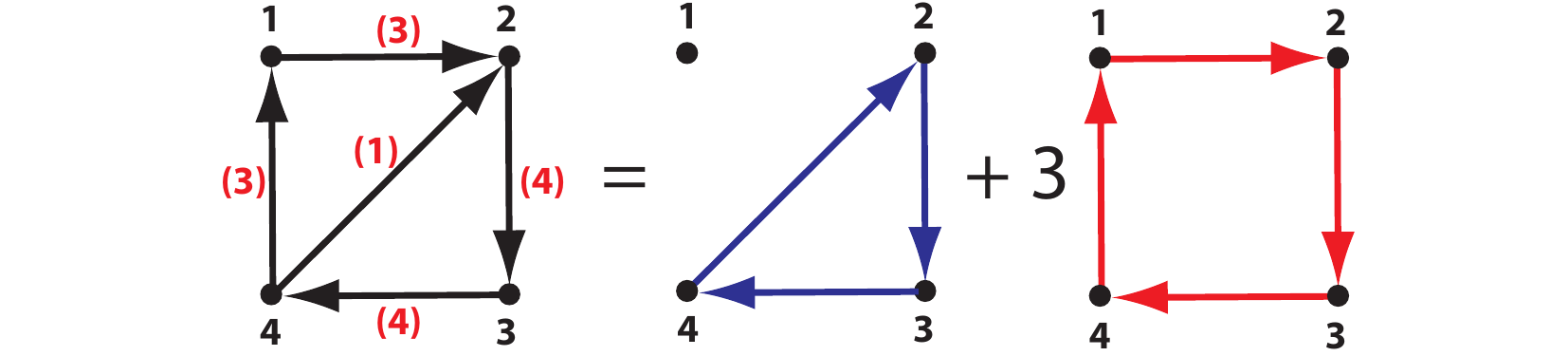}
\end{center}
  \caption{Running example. Computation of the conformal representation of the $z$ of Figure \ref{fig:zeta-omega}\label{fig:zeta-omega2}}
  \end{figure}
\end{example}

\begin{theorem}[The K-ideal is toric] \label{pr:toric}\
  \begin{enumerate}
  \item The K-ideal is the toric ideal of the matrix $A$.
  \item The binomials of the cycles form a Graver basis of the K-ideal.
  \end{enumerate}
\end{theorem}
\begin{proof}
\begin{enumerate}
\item
For each cycle $\omega$ the cycle vector $z(\omega)$ belongs to $\integers(\mathcal D)$. From Equation \eqref{eq:omegavscycle}, $P^{z^+(\omega)} - P^{z^-(\omega)} = P^\omega - P^{r(\omega)}$, therefore the K-ideal is contained in the toric ideal $I(A)$ because of Proposition \ref{prop:cyclebasis}.

To prove the equality we must show that each binomial in $I(A)$ belongs to the K-ideal. From Proposition \ref{pr:graver}.\ref{item:graver1}, it follows that
  \begin{equation*}
    P^{z^+} - P^{z^-} = \prod_{i=1}^n (P^{z^+(\omega_i)})^{\alpha(\omega_i)} -
    \prod_{i=1}^n (P^{z^-(\omega_i)})^{\alpha(\omega_i)}
  \end{equation*}
belongs to the K-ideal.
\item
The Graver basis of a toric ideal is the set of binomials whose exponents are the positive and negative parts of a Graver basis of $\integers(\mathcal D)$. From Propositions \ref{pr:graver} and the previous item the proof follows.
\end{enumerate}
\end{proof}

\subsection{Non-zero K-ideal}
The knowledge that the K-ideal is toric is relevant, because the homomorphism definition in Equation \eqref{eq:h-homo} provides a parametric representation of the the variety. In particular,  the strictly positive $P_a$, $a \in \arcs$, are given by:
\begin{align}
  P_{v\to w} &= s(v,w) \prod_B t_B^{u_{v\to w}(B)} \notag \\ &= s(v,w)
   \prod_{B \colon v\in B, w \notin B } t_B \ \prod_{B \colon w\in B, v \notin B} t_B^{-1},\quad s(v,w)>0, \quad t_B >0 .\label{eq:p-v-to-w}
\end{align}

We observe that the first set of parameters, $s(v,w)$, is a function of the edge, while the second set of parameters, $t_B$, represents the deviation from symmetry. In fact, as
\begin{equation*}
  P_{w\to v} = s(w,v) \prod_B t_B^{u_{w\to v}(B)} = s(v,w) \left(\prod_B t_B^{u_{v\to w}(B)}\right)^{-1},
\end{equation*}
we have $P_{v\to w}  = \left(\prod_B t_B^{u_{v\to w}(B)}\right)^2 P_{w\to v}$, so that $P_{v\to w} = P_{w\to v}$ if, and only if,
\begin{equation*}
  \left(\prod_B t_B^{u_{v\to w}(B)}\right)^2 = 1, \quad v \in V.
\end{equation*}

As the rows of $E$ are linearly independent, the $s(v,w)$'s parameters carry $\#\edges$ degrees of freedom to represent a generic symmetric matrix.  The second set of parameters is not identifiable because the rows of the $U$ matrix are not linearly independent. The parameterization \eqref{eq:p-v-to-w} can be used to derive an explicit form of the invariant probability. All properties of the parameterization are collected in the following Proposition.
\begin{theorem}\label{th:parameters}
Consider the strictly non-zero points on the  K-variety.
\begin{enumerate}
\item
The symmetric parameters $s(e)$, $e\in\edges$, are uniquely determined in Equation \eqref{eq:p-v-to-w}. The parameters $t_{B}$, $\emptyset \neq B \subsetneq V$ are confounded by $\ker U = \set{U^{t}t = 0}$
\item\label{item:param}
An identifiable parameterization is obtained by taking a subset of parameters corresponding to linearly independent rows, denoted by $t_{B}$, $B \subset \mathcal S$:
\begin{equation}\label{eq:p-v-to-w-li}
  P_{v\to w} =  s(v,w) \prod_{B \subset \mathcal S \colon v\in B, w \notin B } t_B \
  \prod_{B \subset \mathcal S \colon w\in B, v \notin B} t_B^{-1}
\end{equation}
\item \label{item:param1}
The detailed balance equations, $ \kappa(v) P_{v\to w}=\kappa(w) P_{w\to v} $, are verified by
\begin{equation}\label{eq:k}
\kappa(v) \propto \prod_{B \colon v\in B} t_B^{-2}
\end{equation}
\end{enumerate}
\end{theorem}
\begin{proof}
\begin{enumerate}
\item We have $\log P = E^{t}s + U^{t}t$ for $P=(P_{v\to w}\colon (v\to w) \in \arcs)$, $s = (s(e) \colon e\in\edges)$, $t = (t_B \colon \emptyset \ne S \subsetneq V)$. If $E^{t}s_1 + U^{t}t_1 = E^{t}s_2 + U^{t}t_2$, then $E^{t}(s_1-s_2) = 0$ because the rows of $E$ are orthogonal to the rows of $U$. Hence, $s_1=s_2$ because $E$ has full rank. Finally, $U^{t}t_1 = U^{t}t_2$.
\item The sub-matrix of $A$ formed by $E$ and by the rows of $U$ in $\mathcal S$ has full rank.
\item
Using Equations \eqref{eq:p-v-to-w}, we have:
\begin{equation*}
\kappa(v) \ s(v,w) \prod_{S \ \colon v\in S, w \notin S } t_B   \prod_{S  \colon w\in S, v \notin S} t_B^{-1} =
\kappa(w) \ s(v,w) \prod_{S \ \colon w\in S, v \notin S } t_B   \prod_{S  \colon v\in S, w \notin S} t_B^{-1}
\end{equation*}
which is equivalent to
\begin{equation*}
\kappa(v)    \prod_{S \ \colon v\in S, w \notin S } t_B^{2}   = \kappa(w)  \prod_{S \ \colon w\in S, v
\notin S } t_B^{2}.
\end{equation*}
By multiplying both terms in the equality by $\prod_{S \ \colon v\in S, w \in S } t_B^{2}$, we obtain
\begin{equation*}
\kappa(v)    \prod_{S \ \colon v\in S} t_B^{2}   = \kappa(w)  \prod_{S \ \colon w\in S} t_B^{2},
\end{equation*}
so that $\kappa(v)=\prod_{S  \colon v\in S} t_B^{-2}$ depends only on $v$ and satisfy the detailed balance condition.
\end{enumerate}
\end{proof}

We are now in the position of stating an algebraic version of Kolmogorov's theorem.

\begin{definition}
The \emph{detailed balance ideal} is the ideal of the ring
\begin{equation*}
  \rationals[\kappa(v):v\in V,P_{v\to w}, (v\to w)\in \arcs]
\end{equation*}
generated by the polynomials
\begin{align*}
  &\prod_{v\in V} \kappa(v) - 1,\\
  &\kappa(v)P_{v\to w} - \kappa(w)P_{v\to w}, \quad (v\to w) \in \arcs .
\end{align*}
\end{definition}
The first polynomial in the list of generators, $\prod_{v\in V} \kappa(v) - 1$, assures that the $\kappa$'s are not zero.
\begin{theorem}\
\begin{enumerate}
\item
A point $P = \left[P_{v\to w}\right]_{v\to w \in \arcs}$ with non-zero components belongs to the variety of the K-ideal if, and only if, there exists $\kappa=\left(\kappa(v) \colon v\in V \right)$ such that $\left(\kappa,P \right)$ belongs to the variety of the detailed balance ideal.
\item
The detailed balance ideal is a toric ideal.
\item
The K-ideal is the $\kappa$-elimination ideal of the detailed balance ideal.
\end{enumerate}
\end{theorem}
\begin{proof}
\begin{enumerate}
\item One direction is the Kolmogorov's theorem. The other direction is a rephrasing of Item \ref{item:param1} of Theorem \ref{th:parameters}.
\item
This ideal is the kernel of the homomorphism defined by \eqref{eq:h-homo}, i.e. $P_{v\to w} \longmapsto s(v,w)
\prod_{B} t_B^{u_{v\to w}(B)}$ together with $\kappa(v) \mapsto \prod_{B  \colon v\in B} t_B^{-2}$.
\item The elimination ideal is generated by dropping the parametric equations of the indeterminates to be eliminated.
\end{enumerate}
\end{proof}

\subsection{Parameterization of reversible transitions}
The parameterization in Theorem \ref{th:parameters}.\ref{item:param} is to be compared to that in Proposition \ref{prop:squarerootpi}. Both split the parameter space into a representation of the invariant probability and a generic symmetric function. The latter is a special case of the former. In fact, it is obtained by the use of the basis of the cocycle space where each $B$ is the set containing one vertex.

The parameterization of reversible Markov matrices is obtained by adding to the representation in Equation
\eqref{eq:p-v-to-w-li} the relevant inequalities. Let us check first the degrees of freedom. As a reversible Markov
matrix supported by a connected graph $\mathcal G$ is parameterized by the joint 2-distributions of the stationary Markov
chain, the number of degrees of freedom is $\# V + \# \edges -1$, i.e. the cyclotomic number of the graph $\mathcal G$.
In the parameterization of Proposition \ref{prop:squarerootpi} the probability $\pi$ carries $\# V -1$ degrees of
freedom, therefore $s$ must carry $\# \edges$ degrees of freedom.
\begin{theorem}\label{th:par-k}
Let $P$ be a matrix with supporting graph $\mathcal G = (V,\edges)$. Let $\mathcal S$ be a family of subsets of $V$ such that the cocycle vectors $u_B$, $B\in \mathcal S$, span the coclycle space.
  \begin{enumerate}
\item $P$ is a reversible Markov matrix if, and only if, there exists a non-negative symmetric function $s\colon V\times V \rightarrow \posreals$ which is supported on the edges $\edges$ and there exist positive parameters $t_B > 0$, $B\in\mathcal S$, such that
\begin{equation}\label{eq:MPR1}
  P_{v\to w} =  s(v,w) \prod_{B \in \mathcal S \colon v\in B, w \notin B } t_B \
  \prod_{B \in \mathcal S \colon w\in B, v \notin B} t_B^{-1}, \quad (v\to w) \in \arcs.
\end{equation}
and the invariant probability is proportional to $\kappa = \prod_{B \ni v} t_B^{-2}$.
\item For $v \in V$,
\begin{equation*}
  \prod_{B \ni v} t_B^{-1} \ge \sum_{w\in N(v)} s(v,w) \prod_{B \ni w} t_B^{-1}.
\end{equation*}
\item The parameters  $s(e)$, $e\in\edges$ and $\kappa(v)$, $v \in V$, are identifiable, while the parameters $t_B$, $B\in\mathcal S$, are identifiable if $u_B$, $B\in \mathcal S$, is a basis of the cocycle space.
\end{enumerate}
\end{theorem}
\begin{proof}
  \begin{enumerate}
  \item Follows from \eqref{eq:p-v-to-w-li} and \eqref{eq:k} and
  \begin{equation*}
 \kappa(w)^{1/2}\kappa(v)^{-1/2}=\frac{\prod_{S  \colon v\in S} t_B}{\prod_{S  \colon w\in S} t_B} = \prod_{S \subset \mathcal S \colon v\in S, w \notin S } t_B \
  \prod_{S \subset \mathcal S \colon w\in S, v \notin S} t_B^{-1}.
  \end{equation*}
\item It follows from Equation \eqref{eq:MPR1} as $P$ is a transition probability.
 \item Assume there exists two set of parameters $t^{(i)}$, $s^{(i)}$, $i=1,2$ giving the same $P$, and define $t=t^{(1)}/t^{(2)}$, $s=s^{(1)}/s^{(2)}$, $\kappa(v) = \kappa^{(1)}(v)/\kappa^{(2)}(v)$. It follows
   \begin{align*}
     1 &= s(v,w) \prod_{B \in \mathcal S \colon v\in B, w \notin B } t_B \
  \prod_{B \in \mathcal S \colon w\in B, v \notin B} t_B^{-1}, \quad (v\to w) \in \arcs,\\
 1 &= s(w,v) \prod_{B \in \mathcal S \colon w \in B, v \notin B } t_B \
  \prod_{B \in \mathcal S \colon v\in B, w \notin B} t_B^{-1}, \quad (w\to v) \in \arcs,
   \end{align*}
hence $s(v,w)s(w,v)= s(v,w)^2 = 1$. In turn we get
   \begin{equation*}
     1 = \left(\prod_{B \in \mathcal S \colon v\in B, w \notin B } t_B
  \prod_{B \in \mathcal S \colon w\in B, v \notin B} t_B^{-1}\right)^2 = \frac{\kappa(v)}{\kappa(w)}.
   \end{equation*}
The identifiability of the $t_B$'s follows from
\begin{equation*}
  \log P_{v\to w} = \log s(v,w) + \sum_{B\in \mathcal S} (\log t_B) u_B
\end{equation*}

 \end{enumerate}
\end{proof}


\begin{example}[Running example continue]
An over-parameterization of two transition probabilities of the $K$-variety is:
\begin{align*}
 P_{3\to 4} &=  s(3,4)\ t_{\{3\}} \ t_{\{1,3\}} \ t_{\{2,3\}} \ t_{\{1,2,3\}} \ t_{\{4\}}^{-1} \ t_{\{1,4\}}^{-1} \
 t_{\{2,4\}}^{-1} \ t_{\{1,2,4\}}^{-1}\\
 P_{4\to 3} &=  s(3,4)\  t_{\{4\}}\ t_{\{1,4\}}\ t_{\{2,4\}}\ t_{\{1,2,4\}}\ t_{\{3\}}^{-1} \ t_{\{1,3\}}^{-1} \
 t_{\{2,3\}}^{-1}  \ t_{\{1,2,3\}}^{-1}
\end{align*}
By choosing the cocycle basis $\mathcal S=\left\{ \{1\}, \{3\},\{1,2\}\right\}$, we have:
\begin{equation*}
  \left\{
\begin{aligned}
\kappa(1)&= t_{\{1\}}^{-2} t_{\{1,2\}}^{-2}  \\
\kappa(2)&= t_{\{1,2\}}^{-2} \\
\kappa(3)&= t_{\{3\}}^{-2} \\
\kappa(4)&= 1
\end{aligned}
\right. \qquad \Longleftrightarrow \qquad
\left\{
\begin{aligned}
t_{\{1\}}&= \kappa(1)^{-1/2} \ \kappa(2)^{1/2}\\
t_{\{3\}}&= \kappa(3)^{-1/2}\\
t_{\{1,2\}} &= \kappa(2)^{-1/2}
\end{aligned}\right.
\end{equation*}

The transition matrix parameterized by $s(e)$, $e \in \edges$ and $t_B$, $S \in \mathcal S$ is
\begin{equation*}
  \bordermatrix[{[}{]}]
  {& 1 & 2 & 3 & 4\cr
  1& \star & s(1,2)\ t_{\{1\}}^{-1} & 0 &s(1,4)\ t_{\{1\}}^{-1} \ t_{\{1,2\}}^{-1}\cr
  2&s(1,2)\ t_{\{1\}} & \star & s(2,3)\ t_{\{1,2\}}^{-1}\ t_{\{3\}} & s(2,4)\ t_{\{1,2\}}^{-1}\cr
  3& 0 &s(2,3)\  t_{\{3\}}^{-1}\ t_{\{1,2\}}  & \star & s(3,4) \ t_{\{3\}}^{-1} \cr
  4& s(1,4) \ t_{\{1\}}\ t_{\{1,2\}} & s(2,4) \  t_{\{1,2\}} &  s(3,4) \ t_{\{3\}} & \star \cr }.
\end{equation*}
where the diagonal terms are uniquely defined if the following set of inequalities is true
\begin{equation*}\left\{
  \begin{aligned}
     1 &\ge s(1,2)  t_{\{1,2\}}+ s(1,4) \cr
     1 &\ge s(1,2)t_{\{1\}} t_{\set{1,2}} + s(2,3) t_{\{3\}} + s(2,4) \cr
     1 &\ge s(2,3) t_{\{1,2\}} + s(3,4) \cr
     1 &\ge s(1,4) t_{\{1\}} t_{\{1,2\}} + s(2,4) t_{\{1,2\}} +  s(3,4) t_{\{3\}} \cr
\end{aligned}\right.
\end{equation*}
\end{example}

\section{Discussion}
\label{sec:discussion}
The algebraic analysis of statistical models of the type $p \propto t^A$ where $A$ is an integer matrix has been first introduced in \cite{geiger|meek|sturmfels:2006}, where an implicit binomial form of the model is derived from the monomial form.

Our analysis differs from that in two respect. First, we move backwards from the binomial form represented by the Kolmogorov's conditions to the monomial form. Second, we parameterize the transition probabilities, so that the normalization requires more than one constant.

Our parameterization of a reversible Markov matrix is based on a generic weight function $s(e)$ on the edges $e\in\edges$ of the structure graph $\mathcal G$ and a monomial form of the invariant probability. When the basis of the cocycle space is given by the vertexes of the graph, the parameterization is identical to that the classical form $P_{v\to w} = \pi(v)^{-1/2} \pi(w)^{1/2} s(v,w)$. The monomial form of the unnormalized invariant probability $\kappa(v)=\prod_{B \in \mathcal B\colon v\in B} t_B^{-2}$ suggests the use of a family of sets $\mathcal B$ smaller than a cocycle basis $\mathcal S$ in order to get a parsimonious statistical model. For example, if the graph $\mathcal G$ is a square $N\times N$ grid, a coarse-grained model could use $n\times n$ sub-grids, $1 < n < N$.

One distinct advantage of the implicit binomial form is its ability to fully describe the closure of its strictly positive part, i.e. the extended exponential family. The computation of an Hilbert basis of the non-negative integer kernel of the cocycle matrix $U$ leads to a parameterization of the extended exponential family as in \cite{malago|pistone:1012.0637}, see also \cite{rauh|kahle|ay:09}. However, in this case the border of the model appears to consist simply on the deletion of edges in the support graph.

It follows from general properties of toric ideals that a Graver basis is a universal Gr\"obner basis and that a universal Gr\"obner basis is a Markov basis, \cite{sturmfels:1996}. The Markov basis property is related with the connectedness of random walks on the fibers of $A$, see \cite{diaconis|sturmfels:98} and subsequent literature on MCMC simulation. In this case it would be a simulation of a random reversible Markov matrix.

Finally, the knowledge of a Graver basis for the K-ideal provides efficient algorithms for discrete optimization, see \cite{deloera|hemmecke|onn|weismantel:2008} and \cite{onn:20xx}.

\subsection*{Acknowledgments}
During this research we have discussed various items with S. Sullivan, B. Sturmfels, G. Casnati and we want to thank them for providing us with supporting advice. Preliminary version of this paper have been presented at the workshop WOGAS2, Warwick University 2010 and the conference CREST-SBM 2, Osaka 2010. We wish to thank G. Letac for bringing reference \cite{suomela:1979} to our attention and F. Rigat for the reference \cite{peskun:1973}. Also they thank S. Onn for pointing out the relevance of his own work on Graver bases. After the publication of the first draft on the ArXiv, Y. A. Mitrophanov provided us with the reference to his own work on the subject, we where not aware of and that with gratefully acknowledge.


\begin{thebibliography}{25}
\bibitem[\protect\citeauthoryear{{CoCoA}Team}{online}]{CocoaSystem}
\begin{bmisc}[author]
\bauthor{\bsnm{{CoCoA}Team},~}
(\byear{online}).
\btitle{{{\hbox{\rm C\kern-.13em o\kern-.07em C\kern-.13em o\kern-.15em A}}}: a
  system for doing {C}omputations in {C}ommutative {A}lgebra}.
\bhowpublished{Available at cocoa.dima.unige.it}.
\end{bmisc}
\endbibitem
\end{thebibliography}

\end{document}